\documentclass[11pt]{article}

\usepackage[margin=1.05in]{geometry}
\usepackage{amsmath,amssymb,amsthm,mathtools}
\usepackage{bm}
\usepackage{mathrsfs}
\usepackage{booktabs}
\usepackage{enumitem}
\usepackage{hyperref}
\usepackage{microtype}

\newtheorem{theorem}{Theorem}[section]
\newtheorem{proposition}[theorem]{Proposition}
\newtheorem{corollary}[theorem]{Corollary}
\newtheorem{lemma}[theorem]{Lemma}
\theoremstyle{definition}
\newtheorem{definition}[theorem]{Definition}
\newtheorem{remark}[theorem]{Remark}

\title{\textbf{Young-lattice diagonals and a doubly graded multiple-zeta decomposition of $e^\gamma$}
}
\author{Ricardo G\'omez-A\'iza}
\date{August 2026}

\begin{document}
\maketitle

\begin{abstract}
An equivalent formulation of the Riemann hypothesis recently led to a
partition expansion naturally indexed by diagonals of the Young lattice.
Segovia isolated the hook families $(r,1^m)$ on these diagonals and
computed their limiting contributions $\rho_r$, while observing that
non-hook families provide a missing contribution.

We introduce a bivariate finite generating function that packages all
Young shapes on every fixed-excess diagonal at once.  For each fixed
$r\geq1$, we obtain a diagonal generating polynomial $D_r(n;z)$ and prove
\[
        A_r(n)\sim C_r\,n\log\log n,
\]
where $C_r$ is the $(r-1)$st coefficient of an explicit convergent
infinite product.  Moreover,
\[
        C_r=\sum_{\nu\vdash r-1} C_\nu ,
\]
giving a canonical decomposition over the partitions of the excess
$r-1$.  The one-part contribution is Segovia's hook constant $\rho_r$,
while the remaining terms give all non-hook corrections simultaneously.

We then refine these constants by introducing coefficients $C_{r,d}$
that record simultaneously the Young-lattice excess $r-1$ and the
number $d$ of non-unit rows.  Row sums recover the fixed-excess constants
$C_r$, while column sums recover the depth decomposition in an
Abel-regularized multiple-zeta expansion of $e^\gamma$.  More precisely,
each partition $\nu\vdash r-1$ is identified with an Abel-regularized
multiple-zeta block of depth $\ell(\nu)$.  Thus the same array
$(C_{r,d})$ organizes the decomposition simultaneously by Young-lattice
excess and multiple-zeta depth.  Our results concern the combinatorial
and asymptotic structure of this decomposition, rather than the Riemann
hypothesis itself.
\end{abstract}

\section{Introduction}

An unexpected connection between the Riemann hypothesis and the
combinatorics of integer partitions has recently emerged from a
criterion proposed by Espinosa.  Starting from classical inequalities
involving the divisor-sum function, Espinosa obtained an equivalent
formulation of the Riemann hypothesis expressed in terms of series
in the harmonic numbers.  His subsequent partition-based identities,
including the Assembly Theorem, provide a natural combinatorial
framework for these expansions.
Segovia gave a rigorous treatment of the resulting criterion and
organized its terms in the Young lattice, introducing the quantities
$E_{i,j}(n)$ and the fixed-diagonal sums that form the starting point
of the present work~\cite{Segovia2026}.
Here we are concerned with the combinatorial and asymptotic
structure that arises from this formulation.  We emphasize from the
outset that our results do not constitute progress toward proving the
Riemann hypothesis.  Rather, we study the partition sums that occur in
the Espinosa--Segovia criterion and show that they can be organized by
a unified generating function.  This point of view leads, in
particular, to a complete description of the fixed diagonals studied
by Segovia and reveals a second grading related to multiple zeta
values.

To describe the problem, let
\[
 E_{i,j}(n)
 =
 \sum_{\substack{\mu\vdash j\\ \ell(\mu)=i}}
 \frac{
 m_\mu(1,1/2,\ldots,1/n)
 }{\mu_1!\cdots\mu_i!},
\]
where $m_\mu$ denotes the monomial symmetric polynomial associated
with the partition $\mu$.  The quantities appearing in Segovia's
diagonal decomposition are
\[
 A_r(n)
 =
 \sum_{i=1}^{n}
 \log(i+r)\,E_{i,i+r-1}(n).
\]
Thus, for fixed $r$, the relevant partitions satisfy
\[
        |\mu|-\ell(\mu)=r-1.
\]
Geometrically, these are the partitions lying on a fixed diagonal of
the Young lattice.

Segovia's analysis follows a particularly simple family along each
such diagonal, namely the hooks.  We write partitions throughout
using parentheses; thus the hook partition
\[
        (r,1^m)
\]
is the shape denoted $[r,1^m]$ in \cite{Segovia2026}.  We use
square brackets only when referring explicitly to Segovia's notation.
He associates to these families constants $\rho_r$ and computes
partial sums of the hook contributions approaching
\[
        1.635904\ldots,
\]
well below $e^\gamma$.  The difference is
accounted for by non-hook partitions.  Already in the first cases,
Segovia identifies families such as
\[
(2,2,1^m),\qquad
(3,2,1^m),\qquad
(2,2,2,1^m)
\]
and points to the need to understand these additional contributions
systematically \cite{Segovia2026}.

There is a simple pattern behind these families that is useful to make
explicit.  Ignoring for the moment the arbitrary trailing parts equal
to $1$, the first few fixed-excess diagonals contain
\[
\begin{array}{c|l}
 r=2 & (2) \\[1mm]
 r=3 & (3),\ (2,2) \\[1mm]
 r=4 & (4),\ (3,2),\ (2,2,2) \\[1mm]
 r=5 & (5),\ (4,2),\ (3,3),\
       (3,2,2),\ (2,2,2,2).
\end{array}
\]
Subtracting $1$ from every non-unit part transforms this array into
\[
\begin{array}{c|l}
 r=2 & (1) \\[1mm]
 r=3 & (2),\ (1,1) \\[1mm]
 r=4 & (3),\ (2,1),\ (1,1,1) \\[1mm]
 r=5 & (4),\ (3,1),\ (2,2),\
       (2,1,1),\ (1,1,1,1).
\end{array}
\]
The entries in the $r$th row are therefore precisely the partitions
\[
        \nu\vdash r-1.
\]
More generally, if
\[
        \nu=(\nu_1,\ldots,\nu_d)\vdash r-1,
\]
then the corresponding family on the original Young-lattice diagonal is
\[
        (\nu_1+1,\ldots,\nu_d+1,1^m),
        \qquad m\geq0.
\]
In particular, Segovia's hook family corresponds to the one-part partition
\[
        \nu=(r-1).
\]

This simple observation suggests that the non-hook terms should not
be treated as a list of exceptional shapes.  They are naturally
indexed by the partitions of the excess $r-1$ itself.

More generally, this is an instance of a familiar principle in
partition enumeration: families that appear unrelated when considered
shape by shape may become considerably simpler once they are encoded
by a common generating function.
This point of view also underlies the unified
treatment of families of partition functions developed in
\cite{AhmadiGomezWard2024}, where integer partitions, Young tableaux,
and permutation cycles are organized within a common combinatorial
framework.  Although the generating functions considered here arise
from a different problem, the same organizing principle motivates our
approach.

Accordingly, rather than continuing the analysis one Young shape at a
time, we package the whole finite problem into the generating function
\begin{equation}\label{eq:Fmaster-intro}
 F_n(u,v)
 =
 \prod_{q=1}^n
 \left(1+v(e^{u/q}-1)\right).
\end{equation}
The two variables have a direct combinatorial meaning: the exponent
of $v$ records the number of parts, while the exponent of $u$ records
their total size.  We prove that
\[
        [u^jv^i]F_n(u,v)=E_{i,j}(n).
\]
Consequently, replacing $v$ by $z/u$ and extracting the coefficient
of $u^{r-1}$ collects an entire fixed-excess diagonal:
\begin{equation}\label{eq:Dr-intro}
 D_r(n;z)
 =
 [u^{r-1}]
 \prod_{q=1}^{n}
 \left(
 1+z\frac{e^{u/q}-1}{u}
 \right)
 =
 \sum_i E_{i,i+r-1}(n)z^i.
\end{equation}
Thus the shape-by-shape Young-lattice calculation is replaced by a
single coefficient extraction.

The finite product \eqref{eq:Dr-intro} also makes the asymptotic
behavior of a fixed diagonal accessible.  After separating the
dominant factor
\[
        \prod_{q=1}^n\left(1+\frac{z}{q}\right),
\]
we are led to the convergent infinite product
\begin{equation}\label{eq:Hcal-intro}
 \mathcal H(u)
 =
 \prod_{q=1}^\infty
 \frac{q}{q+1}
 \left(
 1+\frac{e^{u/q}-1}{u}
 \right).
\end{equation}
Writing
\[
        C_r=[u^{r-1}]\mathcal H(u),
\]
we prove, for every fixed $r$,
\begin{equation}\label{eq:main-intro}
        A_r(n)\sim C_r\,n\log\log n.
\end{equation}
The two factors in this asymptotic have a transparent origin.  The
total mass on the $r$th diagonal is asymptotic to $C_r n$, whereas,
with respect to the corresponding normalized distribution, the
typical number of parts is of order $\log n$.  The logarithmic weight
appearing in $A_r(n)$ therefore contributes a further factor
$\log\log n$.

The constants $C_r$ admit a finer decomposition that retains the
Young-shape information on the corresponding diagonal.  More precisely, we associate a positive constant $C_\nu$ with every
partition $\nu\vdash r-1$ and obtain
\begin{equation}\label{eq:Cpartition-intro}
        C_r=\sum_{\nu\vdash r-1}C_\nu.
\end{equation}
The contribution of the one-part partition is exactly Segovia's hook
constant:
\[
        C_{(r-1)}=\rho_r.
\]
It follows that
\[
        C_r-\rho_r
        =
        \sum_{\substack{\nu\vdash r-1\\
                        \ell(\nu)\geq2}}
        C_\nu.
\]
In this sense, the non-hook terms that first appear individually in
the low diagonals are the beginning of a canonical partition
decomposition of the entire correction.

The partition refinement carries a second natural statistic.  If
\[
        \nu=(\nu_1,\ldots,\nu_d)\vdash r-1,
\]
then
\[
        |\nu|=r-1,
        \qquad
        \ell(\nu)=d.
\]
The first is the Young-lattice excess already used above.  The second
counts the number of non-unit rows in the corresponding Young shape.
Introducing a variable for each statistic leads to the doubly graded
product
\begin{equation}\label{eq:Hscr-intro}
 \mathscr H(u,y)
 =
 \prod_{q=1}^\infty
 \left[
 1+
 y\,\frac{q(e^{u/q}-1)-u}{u(q+1)}
 \right].
\end{equation}
We define
\[
        C_{r,d}
        =
        [u^{r-1}y^d]\mathscr H(u,y).
\]
The coefficients $C_{r,d}$ form a triangular array:
\[
\begin{array}{c|ccccc}
 & d=1 & d=2 & d=3 & d=4 & \cdots \\ \hline
 r=2 & C_{2,1} \\
 r=3 & C_{3,1} & C_{3,2} \\
 r=4 & C_{4,1} & C_{4,2} & C_{4,3} \\
 r=5 & C_{5,1} & C_{5,2} & C_{5,3} & C_{5,4}\\
 \vdots & \vdots & \vdots & \vdots & \vdots & \ddots
\end{array}
\]
and the first column is precisely
\[
        C_{r,1}=\rho_r.
\]

This array provides the point at which two apparently different
decompositions meet.  Reading a row gives
\[
        C_r=\sum_d C_{r,d},
\]
and therefore collects all Young shapes of a fixed excess.  Reading a
column instead gives
\[
        S_d=\sum_r C_{r,d},
\]
which collects contributions with a fixed number of non-unit rows.
We show that this second statistic is exactly the depth appearing in
Espinosa's multiple-zeta-value expansion of $e^\gamma$.

The identification with multiple zeta values requires some care,
since the relevant alternating expansions are not absolutely
convergent at the endpoint.  We handle this by introducing an Abel
regulator: the multiple zeta values (MZV) identities are first established in an absolutely
convergent regime and the regulator is then allowed to approach its
endpoint.  This gives, in particular, a rigorous identification of
each $C_\nu$ with an Abel-regularized MZV block of depth
$d=\ell(\nu)$:
\[
\boxed{
\begin{array}{c}
\text{partition }\nu = (\nu_1,\ldots , \nu_d)
\\[1mm]
\updownarrow\\[-1mm]
\text{Young family }
(\nu_1+1,\ldots,\nu_d+1,1^m)
\\[1mm]
\updownarrow\\[-1mm]
\text{MZV block of depth }d .
\end{array}}
\]
The two coordinates of the array $C_{r,d}$ consequently record
Young-lattice excess and MZV depth simultaneously.

Specializing \eqref{eq:Hscr-intro} at $u=y=1$ gives
\[
 \mathscr H(1,1)
 =
 \prod_{q\geq1}
 \frac{q}{q+1}e^{1/q}
 =
 e^\gamma.
\]
Since the coefficients $C_{r,d}$ are nonnegative, Tonelli's theorem
shows that the row and column decompositions are two legitimate
orders of summation of the same positive mass:
\[
 e^\gamma
 =
 1+\sum_{r\geq2}\sum_{d=1}^{r-1}C_{r,d}
 =
 1+\sum_{d\geq1}\sum_{r\geq d+1}C_{r,d}.
\]
The first order is the fixed-excess decomposition suggested by the
Young lattice, whereas the second is the fixed-depth decomposition
associated with the MZV expansion.

Related connections between partition-theoretic structures and zeta
values occur in the theory of partition zeta functions developed by
Ono, Rolen, and Schneider~\cite{OnoRolenSchneider}, where sums over
integer partitions lead to identities involving the Riemann zeta
function and multiple zeta values.  The constructions considered here
are different in origin: our grading is determined by partition size
and length, rather than by the multiplicative partition norms used in
partition zeta functions.

The appearance of multiple zeta values in this setting is not new.
Espinosa formulated an Abel-regularized MZV decomposition of
$e^\gamma$, and Hucht subsequently organized its depth contributions
by a generating function involving
\[
        \frac{e^{1/q}}{1+1/q}-1.
\]
Our contribution is instead to retain the additional excess grading.
The product \eqref{eq:Hscr-intro} simultaneously records the
fixed-excess Young-lattice decomposition and the fixed-depth MZV
decomposition, while \eqref{eq:main-intro} provides the asymptotic
behavior of each fixed Young-lattice diagonal.  In particular, the
global identity $e^\gamma$ is used here as a consistency relation and
not claimed as a new result.

The remainder of the paper develops these constructions in order.
Section~2 derives the finite generating function and the fixed-diagonal
extraction, and Section~3 establishes the fixed-excess asymptotics.
Sections~4--6 develop the partition and depth refinements, while
Section~7 gives the regulated MZV identification.  We conclude with
computational checks and a discussion of the resulting
two-dimensional structure.

\section{From partition sums to fixed-excess diagonals}
\label{sec:finite}

We begin by deriving the finite generating function that will be used
throughout the paper.  The construction is elementary, but it is worth
making explicit how the partition structure enters.

For a partition
\[
        \mu=(\mu_1,\ldots,\mu_i)\vdash j,
\]
let $m_\mu(x_1,\ldots,x_n)$ denote the associated monomial symmetric
polynomial: the sum of all distinct monomials obtained by assigning
the positive exponents $\mu_1,\ldots,\mu_i$ to distinct variables
among $x_1,\ldots,x_n$.  We use the quantities
\[
 E_{i,j}(n)
 =
 \sum_{\substack{\mu\vdash j\\ \ell(\mu)=i}}
 \frac{
 m_\mu(1,1/2,\ldots,1/n)
 }{\mu_1!\cdots\mu_i!},
\]
following the notation of Segovia \cite{Segovia2026}.  We also set
$E_{0,0}(n)=1$.

The two indices have a simple interpretation.  The index $j$ is the
size of the partition, while $i$ is its number of parts.  Hence
\[
        j-i=\sum_{k=1}^{i}(\mu_k-1)
\]
measures the excess of the partition over the all-unit partition
$(1^i)$.  Fixing $j-i=r-1$ is precisely the fixed-excess condition
that defines the $r$th diagonal considered in the Introduction.
The terminology ``diagonal'' is most transparent after passing from
individual partition shapes to their size and length.  Figure~\ref{fig:diagonals}
shows the same fixed-excess families both in the Young lattice and under
the projection
\[
        \mu\longmapsto\bigl(\ell(\mu),|\mu|\bigr).
\]
In the latter representation the condition
$|\mu|-\ell(\mu)=r-1$ becomes a literal diagonal.

\begin{figure}[htbp] %  figure placement: here, top, bottom, or page
   \centering
   \includegraphics[width=6.3in]{FigDiagonals.pdf} 
   \caption{\textbf{Fixed-excess diagonals in the Young lattice and their
shape-family decomposition.}
(a) Partitions are shown in the Young lattice, graded vertically by their
size $|\mu|$. For fixed excess
$r-1=|\mu|-\ell(\mu)$, the corresponding ``diagonal'' is not a single
path in the Young lattice, but a union of rays indexed by partitions
$\nu\vdash r-1$. Each ray begins at
$(\nu_1+1,\ldots,\nu_d+1)$ and continues by appending unit parts.
(b) Under the projection
$\mu\mapsto(\ell(\mu),|\mu|)$, partitions having the same size and
number of parts occupy the same cell, and each fixed-excess family
becomes a literal diagonal $|\mu|-\ell(\mu)=r-1$.
(c) Resolving each cell according to its Young shape separates these
diagonals into the individual rays indexed by
$\nu\vdash r-1$; along the ray associated with
$\nu=(\nu_1,\ldots,\nu_d)$, the partitions are
$(\nu_1+1,\ldots,\nu_d+1,1^m)$, $m\geq0$.
Colors identify the same fixed-excess classes throughout the three
panels.}
\label{fig:diagonals}
\end{figure}

This representation also explains the coefficient extraction used below.
Indeed, if $i=\ell(\mu)$ and $j=|\mu|$, then the $r$th diagonal is
characterized by $j-i=r-1$.

To see how all these quantities can be generated simultaneously,
consider one of the variables in the specialization
\[
        (1,1/2,\ldots,1/n),
\]
say $1/q$.  If this variable is not used in a monomial, it contributes
$1$.  If it is used with a positive exponent $m$, its contribution,
including the factorial appearing in the definition of $E_{i,j}(n)$,
is
\[
        v\,\frac{u^m}{m!\,q^m}.
\]
Here $u$ records the size contributed by the exponent and $v$ records
the fact that one part has been selected.  Summing over all possible
positive exponents gives
\[
        v\sum_{m\geq1}\frac{u^m}{m!\,q^m}
        =
        v(e^{u/q}-1).
\]
Thus the contribution associated with the index $q$ is
\[
        1+v(e^{u/q}-1).
\]
Since the indices $q=1,\ldots,n$ are selected independently, this
immediately suggests the product
\[
        F_n(u,v)
        =
        \prod_{q=1}^n
        \left(1+v(e^{u/q}-1)\right).
\]

The following proposition makes this observation precise.

\begin{proposition}[Finite master generating function]
\label{prop:F}
For every $n\geq1$,
\begin{equation}\label{eq:Fmaster}
 F_n(u,v)
 =
 \prod_{q=1}^n
 \left(1+v(e^{u/q}-1)\right)
 =
 \sum_{i,j\geq0}E_{i,j}(n)v^iu^j.
\end{equation}
In particular,
\[
        [u^jv^i]F_n(u,v)=E_{i,j}(n).
\]
\end{proposition}

\begin{proof}
Expanding an individual factor gives
\[
 1+v(e^{u/q}-1)
 =
 1+v\sum_{m\geq1}\frac{u^m}{m!\,q^m}.
\]
To obtain a term containing $v^i$, we choose a nonconstant term from
exactly $i$ distinct indices
\[
        q_1,\ldots,q_i\in\{1,\ldots,n\}.
\]
If the corresponding positive exponents are
$m_1,\ldots,m_i$, the resulting term is
\[
 v^i u^{m_1+\cdots+m_i}
 \frac{1}
 {m_1!\cdots m_i!\,
  q_1^{m_1}\cdots q_i^{m_i}}.
\]
Fixing the total degree
\[
        m_1+\cdots+m_i=j
\]
and grouping the exponent vectors according to the partition
$\mu\vdash j$ obtained by sorting $(m_1,\ldots,m_i)$ produces exactly
the distinct monomials in
\[
        m_\mu(1,1/2,\ldots,1/n).
\]
The factorial denominator is unchanged under permutation of the
parts.  Summing over all partitions of $j$ having $i$ parts therefore
gives $E_{i,j}(n)$.
\end{proof}

It may be useful to see the grouping in one small case.  Consider the
coefficient of $u^4v^2$.  The partitions of $4$ with two parts are
\[
        (3,1)\qquad\text{and}\qquad(2,2),
\]
and Proposition~\ref{prop:F} gives
\[
 [u^4v^2]F_n(u,v)
 =
 \frac{m_{(3,1)}(1,1/2,\ldots,1/n)}{3!}
 +
 \frac{m_{(2,2)}(1,1/2,\ldots,1/n)}{2!\,2!}.
\]
The first term collects choices of two distinct indices carrying
exponents $3$ and $1$, in either order; the second collects choices
in which both selected indices carry exponent $2$.  Thus the
symmetric-polynomial notation is exactly what remains after the
individual assignments of exponents to distinct indices have been
grouped by partition shape.

\subsection{Extracting a fixed diagonal}

The master product~\eqref{eq:Fmaster} keeps track separately of the
size $j$ and the number of parts $i$.  To isolate a fixed excess
$j-i=r-1$, substitute
\[
        v=\frac{z}{u}.
\]
Then
\[
        F_n(u,z/u)
        =
        \sum_{i,j\geq0}E_{i,j}(n)z^iu^{j-i}.
\]
Thus the exponent of $u$ records the excess, while $z$ continues to
record the number of parts.

\begin{definition}
For $r\geq1$, define
\[
 D_r(n;z)
 :=
 \sum_{i\geq0}E_{i,i+r-1}(n)z^i.
\]
For $r\geq2$ the sum necessarily begins at $i=1$.  When $r=1$ our
convention $E_{0,0}(n)=1$ contributes the additional constant term
$1$.
\end{definition}

\begin{proposition}[Complete diagonal extraction]
\label{prop:diagonal}
For every $r\geq1$,
\begin{equation}\label{eq:Dr}
 D_r(n;z)
 =
 [u^{r-1}]
 \prod_{q=1}^n
 \left(
 1+z\frac{e^{u/q}-1}{u}
 \right).
\end{equation}
\end{proposition}

\begin{proof}
By Proposition~\ref{prop:F},
\[
 F_n(u,z/u)
 =
 \sum_{i,j\geq0}E_{i,j}(n)z^iu^{j-i}.
\]
Extracting the coefficient of $u^{r-1}$ imposes
$j-i=r-1$, and hence
\[
 [u^{r-1}]F_n(u,z/u)
 =
 \sum_{i\geq0}E_{i,i+r-1}(n)z^i
 =
 D_r(n;z).
\]
\end{proof}

There are two useful pieces of information in $D_r(n;z)$.  Setting
$z=1$ gives the total unweighted mass of the $r$th diagonal,
\[
        D_r(n;1)
        =
        \sum_iE_{i,i+r-1}(n),
\]
whereas retaining $z$ records how this mass is distributed according
to the number of parts $i$.  This second role of $z$ will become
important in Section~\ref{sec:asymptotics}: after normalization,
$D_r(n;z)$ will be the probability generating function of the number
of parts on a fixed diagonal.

Proposition~\ref{prop:diagonal} also makes explicit the
Young-shape parametrization described in the Introduction.  If
\[
        \mu=(\mu_1,\ldots,\mu_i)
\]
occurs on the $r$th diagonal, then
\[
        \sum_{k=1}^i(\mu_k-1)=r-1.
\]
Deleting the zero terms after subtracting $1$ from every part gives
a partition
\[
        \nu\vdash r-1.
\]
Conversely, if
\[
        \nu=(\nu_1,\ldots,\nu_d)\vdash r-1,
\]
then every partition in the corresponding family has the form
\[
        (\nu_1+1,\ldots,\nu_d+1,1^m),
        \qquad m\geq0.
\]
We record this correspondence for later use.

\begin{corollary}[Partition indexing of a diagonal]
\label{cor:shape}
Fix $r\geq2$.  The Young-shape families occurring on the diagonal
$j-i=r-1$ are in one-to-one correspondence with the partitions
$\nu\vdash r-1$.  Under this correspondence,
\[
        \nu=(\nu_1,\ldots,\nu_d)
        \quad\longleftrightarrow\quad
        (\nu_1+1,\ldots,\nu_d+1,1^m),
        \qquad m\geq0.
\]
\end{corollary}

Thus all the Young-shape families on a fixed diagonal are already
present in a single coefficient extraction.  This allows us to study
their total asymptotic mass without calculating the individual
families separately.

\section{Fixed-excess asymptotics}
\label{sec:asymptotics}

For fixed $r$, the quantity
\[
        D_r(n;1)=\sum_i E_{i,i+r-1}(n)
\]
is the total mass carried by the $r$th diagonal.  The finite product
from Section~2 separates naturally into a classical gamma-function
factor, which carries the leading dependence on $n$, and a convergent
infinite product containing the fixed-excess constants.  The marker
$z$, meanwhile, records the number of parts and allows us to locate
the mass of the diagonal around $i\asymp\log n$.  These two features
will account for the scale $n\log\log n$ in the asymptotics of
$A_r(n)$.

\subsection{Separating the dominant factor}

Recall from Proposition~\ref{prop:diagonal} that
\[
 D_r(n;z)
 =
 [u^{r-1}]
 \prod_{q=1}^n
 \left(
 1+z\frac{e^{u/q}-1}{u}
 \right).
\]
It is useful to write
\[
        \phi_q(u):=\frac{e^{u/q}-1}{u},
\]
where the apparent singularity at $u=0$ is removable and
$\phi_q(0)=1/q$.  Thus
\[
 D_r(n;z)
 =
 [u^{r-1}]
 \prod_{q=1}^n(1+z\phi_q(u)).
\]

For large $q$,
\[
        \phi_q(u)
        =
        \frac1q+\frac{u}{2q^2}+O(q^{-3}),
\]
uniformly for $u$ in compact sets.  The term $1/q$ is responsible for
the leading growth of the product.  We therefore factor it out:
\begin{equation}\label{eq:factor}
 \prod_{q=1}^n(1+z\phi_q(u))
 =
 \prod_{q=1}^n\left(1+\frac zq\right)
 R_n(u,z),
\end{equation}
where
\[
 R_n(u,z)
 :=
 \prod_{q=1}^n
 \frac{1+z\phi_q(u)}{1+z/q}.
\]
The first factor is explicit:
\begin{equation}\label{eq:gamma}
 \prod_{q=1}^n\left(1+\frac zq\right)
 =
 \frac{\Gamma(n+1+z)}
 {\Gamma(1+z)\Gamma(n+1)}.
\end{equation}
The second factor has a finite limit as $n\to\infty$.

\begin{lemma}[Normal convergence]\label{lem:normal}
There exist neighborhoods $U$ of $u=0$ and $V$ of $z=1$ such that
$R_n(u,z)$ converges locally uniformly on $U\times V$ to
\[
 R(u,z)
 =
 \prod_{q=1}^\infty
 \frac{1+z\phi_q(u)}{1+z/q}.
\]
The convergence is normal.  In particular, coefficient extraction
in $u$ and differentiation finitely many times with respect to $z$
commute with the limit.
\end{lemma}

\begin{proof}
Uniformly for $u$ in a compact disk,
\[
 \phi_q(u)
 =
 \frac1q+\frac{u}{2q^2}+O(q^{-3}).
\]
For $z$ in a compact neighborhood of $1$ avoiding the poles
$z=-q$, it follows that
\[
 \frac{1+z\phi_q(u)}{1+z/q}
 =
 1+O(q^{-2}),
\]
uniformly on compact subsets of $U\times V$.  Hence the series of
deviations from $1$,
\[
 \sum_{q\geq1}
 \left|
 \frac{1+z\phi_q(u)}{1+z/q}-1
 \right|,
\]
converges uniformly on compact subsets of $U\times V$.  The standard
convergence theorem for infinite products therefore gives locally
uniform convergence of $R_n(u,z)$ to a holomorphic limit $R(u,z)$.
Cauchy's integral formula then implies that coefficient extraction
in $u$ and differentiation finitely many times with respect to $z$
commute with the limit.
\end{proof}

At $z=1$, the limiting product takes a particularly simple form.
We write
\begin{equation}\label{eq:Hcal}
 \mathcal H(u)
 :=
 R(u,1)
 =
 \prod_{q=1}^\infty
 \frac{q}{q+1}
 \left(
 1+\frac{e^{u/q}-1}{u}
 \right),
\end{equation}
and define
\begin{equation}\label{eq:Crdef}
        C_r:=[u^{r-1}]\mathcal H(u).
\end{equation}
These constants will be the central objects in the remainder of the
paper.  For the moment, they arise simply as the limiting
coefficients left after the universal gamma factor has been removed.

More generally, for $z$ near $1$, set
\[
 K_r(z)
 :=
 \frac{1}{\Gamma(1+z)}
 [u^{r-1}]R(u,z).
\]
Notice that $C_r>0$ for every $r\geq1$.  Hence
$K_r(1)=C_r\neq0$, and after shrinking the neighborhood of
$z=1$ if necessary, $K_r(z)$ has no zeros there.  We may
therefore use a holomorphic branch of $\log K_r(z)$ in the
moment calculation below.

\begin{theorem}[Asymptotic mass of a fixed diagonal]
\label{thm:diagonalmass}
For each fixed $r\geq1$,
\begin{equation}\label{eq:Drasymp}
        D_r(n;z)
        =
        n^z\bigl(K_r(z)+o(1)\bigr)
\end{equation}
locally uniformly for $z$ in a neighborhood of $1$.  The asymptotic
remains valid after any fixed number of differentiations with respect
to $z$.  In particular,
\begin{equation}\label{eq:Drmass}
        D_r(n;1)\sim C_r n.
\end{equation}
\end{theorem}

\begin{proof}
Combining \eqref{eq:factor} and \eqref{eq:gamma} gives
\[
 D_r(n;z)
 =
 \frac{\Gamma(n+1+z)}
 {\Gamma(1+z)\Gamma(n+1)}
 [u^{r-1}]R_n(u,z).
\]
By Lemma~\ref{lem:normal},
\[
 [u^{r-1}]R_n(u,z)
 =
 [u^{r-1}]R(u,z)+o(1)
\]
locally uniformly near $z=1$.  On compact subsets of this
neighborhood, the gamma-ratio asymptotic
\[
 \frac{\Gamma(n+1+z)}{\Gamma(n+1)}
 =
 n^z\bigl(1+O(n^{-1})\bigr)
\]
is uniform.  Substitution yields \eqref{eq:Drasymp}.
The convergence above is locally uniform in a complex neighborhood
of $z=1$, and all functions involved are holomorphic there.
Cauchy's integral formula therefore implies local uniform convergence
of every fixed $z$-derivative on a slightly smaller neighborhood.
Thus the asymptotic in \eqref{eq:Drasymp} may be differentiated any fixed
number of times with respect to $z$. In particular, at
$z=1$ we have $\Gamma(2)=1$ and
\[
        K_r(1)=C_r,
\]
which proves \eqref{eq:Drmass}.
\end{proof}

Theorem~\ref{thm:diagonalmass} already gives one half of the
asymptotic behavior of $A_r(n)$.  For every fixed excess, the total
mass of the corresponding diagonal grows linearly with $n$; the
dependence on $r$ is contained entirely in the coefficient $C_r$.
What remains is to understand the typical size of the index $i$ that
carries this mass.

\subsection{Where is the mass on a fixed diagonal?}

The variable $z$ in $D_r(n;z)$ was introduced in Section~\ref{sec:finite}
to mark the number of parts.  Since the coefficients
$E_{i,i+r-1}(n)$ are nonnegative, normalizing them gives a probability
distribution.  For fixed $r\geq1$, define a random variable
$I_{n,r}$ by
\[
 \mathbb P(I_{n,r}=i)
 =
 \frac{E_{i,i+r-1}(n)}
 {D_r(n;1)}.
\]
Its probability generating function is therefore
\begin{equation}\label{eq:Pnr}
 P_{n,r}(z)
 :=
 \mathbb E[z^{I_{n,r}}]
 =
 \frac{D_r(n;z)}{D_r(n;1)}.
\end{equation}

Theorem~\ref{thm:diagonalmass} now has a probabilistic interpretation:
near $z=1$,
\[
 P_{n,r}(z)
 \sim
 n^{z-1}\frac{K_r(z)}{K_r(1)}.
\]
The factor $n^{z-1}$ is the probability generating function of a
Poisson random variable with mean $\log n$.  We only need the first
two moments here, but this observation explains why the number of
parts is concentrated around $\log n$.

\begin{lemma}[Concentration of the number of parts]
\label{lem:concentration}
For every fixed $r\geq1$,
\[
 \mathbb E[I_{n,r}]
 =
 \log n+O(1),
 \qquad
 \operatorname{Var}(I_{n,r})
 =
 \log n+O(1).
\]
Consequently,
\[
        \frac{I_{n,r}}{\log n}
        \longrightarrow1
\]
in $L^2$ and hence in probability.
\end{lemma}

\begin{proof}
From Theorem~\ref{thm:diagonalmass},
\[
 \log D_r(n;z)
 =
 z\log n+\log K_r(z)+o(1)
\]
locally uniformly near $z=1$.  Differentiating gives
\[
 \mathbb E[I_{n,r}]
 =
 \frac{D_r'(n;1)}{D_r(n;1)}
 =
 \log n+\frac{K_r'(1)}{K_r(1)}+o(1).
\]
For a probability generating function,
\[
 \operatorname{Var}(I_{n,r})
 =
 (\log D_r)''(1)+(\log D_r)'(1),
\]
and therefore
\[
        \operatorname{Var}(I_{n,r})
        =
        \log n+O(1).
\]
It follows that
\begin{align*}
\mathbb E\left[
\left(
\frac{I_{n,r}}{\log n}-1
\right)^2
\right]
&=
\frac{\operatorname{Var}(I_{n,r})}{(\log n)^2}
+
\left(
\frac{\mathbb E[I_{n,r}]}{\log n}-1
\right)^2\\
&=
O\left(\frac{1}{\log n}\right).
\end{align*}
Consequently,
\[
        \frac{I_{n,r}}{\log n}
        \longrightarrow 1
\]
in $L^2$, and hence also in probability.
\end{proof}

Thus a typical partition contributing to a fixed-excess diagonal has
approximately $\log n$ parts.  Since the weight in $A_r(n)$ is
$\log(i+r)$, we should therefore expect it to contribute
$\log\log n$.  The following elementary lemma makes this precise.

\begin{lemma}[The logarithmic weight]\label{lem:logweight}
For every fixed $r\geq1$,
\[
        \mathbb E[\log(I_{n,r}+r)]
        \sim\log\log n.
\]
\end{lemma}

\begin{proof}
By Jensen's inequality,
\[
 \mathbb E[\log(I_{n,r}+r)]
 \leq
 \log(\mathbb E[I_{n,r}]+r)
 =
 \log\log n+o(1).
\]
For the reverse inequality, fix $0<\varepsilon<1$.  By
Lemma~\ref{lem:concentration},
\[
 \mathbb P\bigl(
 I_{n,r}\geq(1-\varepsilon)\log n
 \bigr)
 \longrightarrow1.
\]
Hence
\begin{align*}
 \mathbb E[\log(I_{n,r}+r)]
 &\geq
 \mathbb P\bigl(
 I_{n,r}\geq(1-\varepsilon)\log n
 \bigr)\\
 &\qquad{}\times
 \log\bigl((1-\varepsilon)\log n+r\bigr).
\end{align*}
After division by $\log\log n$, the right-hand side tends to $1$.
Together with the upper bound this proves the result.
\end{proof}

\begin{theorem}[Fixed-excess asymptotic]\label{thm:main}
For every fixed $r\geq1$,
\[
        A_r(n)\sim C_r\,n\log\log n.
\]
\end{theorem}

\begin{proof}
By the definition of $I_{n,r}$,
\[
A_r(n)
=
D_r(n;1)\,
\mathbb E[\log(I_{n,r}+r)].
\]
By Theorem~\ref{thm:diagonalmass},
\[
D_r(n;1)\sim C_r n,
\]
while Lemma~\ref{lem:logweight} gives
\[
\mathbb E[\log(I_{n,r}+r)]
\sim\log\log n.
\]
Multiplying the two asymptotics proves the claim.
\end{proof}

The proof also explains the normalization in
Theorem~\ref{thm:main}.  It is not an accidental scale introduced to
obtain a finite limit.  Rather,
\[
 \underbrace{n}_{\text{mass of the diagonal}}
 \times
 \underbrace{\log\log n}_{\text{typical logarithmic weight}}
\]
arises from two distinct features of the same generating function.
The coefficient variable $u$ isolates the fixed excess and produces
the constant $C_r$, while the marker $z$ locates the number of parts
around $\log n$.

\begin{remark}
The stronger marked asymptotic
\[
        D_r(n;z)
        \sim n^zK_r(z)
\]
contains more information than is needed for
Theorem~\ref{thm:main}.  Its form suggests a mod-Poisson description
of $I_{n,r}$ and, consequently, a central limit theorem after
centering by $\log n$ and scaling by $\sqrt{\log n}$.  We leave this
probabilistic refinement for future work.
\end{remark}

\section{Partition refinement and the non-hook corrections}
\label{sec:partition-refinement}

Theorem~\ref{thm:main} reduces the asymptotic behavior of the $r$th
diagonal to the coefficient
\[
        C_r=[u^{r-1}]\mathcal H(u).
\]
Its decomposition reflects precisely the partition indexing introduced
in Section~2: every partition $\nu\vdash r-1$ contributes one term to
$C_r$, corresponding to the Young-shape family
\[
        (\nu_1+1,\ldots,\nu_d+1,1^m),\qquad m\geq0.
\]
This gives a precise way of separating Segovia's hook contribution
from the remaining shapes.

Recall that
\[
 \mathcal H(u)
 =
 \prod_{q=1}^\infty
 \frac{q}{q+1}
 \left(
 1+\frac{e^{u/q}-1}{u}
 \right).
\]
For an individual factor we have
\begin{align*}
 \frac{q}{q+1}
 \left(
 1+\frac{e^{u/q}-1}{u}
 \right)
 &=
 \frac{q}{q+1}
 \left(
 1+\frac1q+
 \sum_{k\geq1}
 \frac{u^k}{(k+1)!\,q^{k+1}}
 \right)\\
 &=
 1+
 \sum_{k\geq1}
 \frac{u^k}{(k+1)!\,q^k(q+1)}.
\end{align*}
It is therefore convenient to set
\begin{equation}\label{eq:ak}
        a_k(q)
        :=
        \frac{1}{(k+1)!\,q^k(q+1)},
        \qquad k,q\geq1.
\end{equation}
Then
\begin{equation}\label{eq:Hcal-ak}
        \mathcal H(u)
        =
        \prod_{q\geq1}
        \left(
        1+\sum_{k\geq1}a_k(q)u^k
        \right).
\end{equation}

The combinatorics of this product is now transparent.  To obtain a
term of degree $r-1$, choose nonconstant terms from some number $d$
of distinct factors and choose positive exponents
\[
        k_1,\ldots,k_d
        \qquad\text{with}\qquad
        k_1+\cdots+k_d=r-1.
\]
After arranging these exponents in decreasing order, they form a
partition
\[
        \nu=(\nu_1,\ldots,\nu_d)\vdash r-1.
\]
Thus the coefficient $C_r$ naturally decomposes according to the
partitions of $r-1$.

For a partition $\nu$, let $m_j(\nu)$ denote the multiplicity of the
part $j$ in $\nu$.

\begin{definition}
Let
\[
        \nu=(\nu_1,\ldots,\nu_d)\vdash r-1.
\]
We define its contribution to the $r$th diagonal by
\begin{equation}\label{eq:Cnu}
 C_\nu
 :=
 \frac{1}{\prod_{j\geq1}m_j(\nu)!}
 \sum_{\substack{q_1,\ldots,q_d\geq1\\
                 q_a\neq q_b\;(a\neq b)}}
 \prod_{s=1}^d a_{\nu_s}(q_s).
\end{equation}
\end{definition}

Since $a_k(q)=O(q^{-k-1})$ for every $k\geq1$, the sum in
\eqref{eq:Cnu} converges absolutely; in particular, $C_\nu>0$.
The multiplicity factor in \eqref{eq:Cnu} accounts for the fact that
the indices $q_1,\ldots,q_d$ are ordered, whereas the parts of
$\nu$ form a multiset.  For a fixed unordered set of $d$ distinct
indices, the ordered sum produces $d!$ assignments.  If the part
$j$ occurs $m_j(\nu)$ times, permutations among those equal parts
do not produce distinct assignments.  Hence the number of distinct
assignments is
\[
        \frac{d!}{\prod_{j\geq1}m_j(\nu)!},
\]
which is exactly what the factor
$\bigl(\prod_j m_j(\nu)!\bigr)^{-1}$ produces from the ordered
sum in \eqref{eq:Cnu}.

\begin{theorem}[Partition decomposition]
\label{thm:partitiondecomp}
For every $r\geq2$,
\begin{equation}\label{eq:Cr-partition}
        C_r
        =
        \sum_{\nu\vdash r-1}C_\nu.
\end{equation}
Under the correspondence of Corollary~\ref{cor:shape}, the term
$C_\nu$, for
\[
        \nu=(\nu_1,\ldots,\nu_d),
\]
is associated with the Young-shape family
\[
        (\nu_1+1,\ldots,\nu_d+1,1^m),
        \qquad m\geq0.
\]
\end{theorem}

\begin{proof}
Expand \eqref{eq:Hcal-ak}.  A contribution to $u^{r-1}$ is obtained
by selecting nonconstant terms from $d$ distinct factors, with
positive exponents whose sum is $r-1$.  Grouping these selections
according to the partition
\[
        \nu\vdash r-1
\]
formed by the selected exponents gives \eqref{eq:Cnu}.  The
multiplicity factorials account for repeated equal parts, as
described above.  Summing over all partitions of $r-1$ gives
\eqref{eq:Cr-partition}.  The correspondence with the Young-shape
families is exactly Corollary~\ref{cor:shape}.
\end{proof}

The one-part partition $\nu=(r-1)$ gives
\[
        C_{(r-1)}
        =
        \frac1{r!}
        \sum_{q\geq1}\frac1{q^{r-1}(q+1)}.
\]

\subsection{The hook term}

The decomposition \eqref{eq:Cr-partition} singles out one particularly
simple partition,
\[
        \nu=(r-1).
\]
It corresponds to the Young-shape family
\[
        (r,1^m),
\]
and hence to the hook family studied by Segovia.

For this one-part partition, \eqref{eq:Cnu} reduces to
\[
 C_{(r-1)}
 =
 \sum_{q\geq1}a_{r-1}(q)
 =
 \frac1{r!}
 \sum_{q\geq1}
 \frac1{q^{r-1}(q+1)}.
\]

\begin{corollary}[Segovia's hook component]
\label{cor:hook}
For every $r\geq2$,
\begin{equation}\label{eq:hook}
        C_{(r-1)}=\rho_r,
\end{equation}
where
\[
 \rho_r
 =
 \frac1{r!}
 \left[
 (-1)^r+
 \sum_{j=3}^{r}
 (-1)^{r+j}\zeta(j-1)
 \right]
\]
is Segovia's hook constant.  Consequently,
\begin{equation}\label{eq:correction}
        C_r-\rho_r
        =
        \sum_{\substack{\nu\vdash r-1\\
                        \ell(\nu)\geq2}}
        C_\nu.
\end{equation}
\end{corollary}

\begin{proof}
The first equality follows from
\[
 C_{(r-1)}
 =
 \frac1{r!}
 \sum_{q\geq1}
 \frac1{q^{r-1}(q+1)}.
\]
The elementary identity
\[
 \sum_{q\geq1}
 \frac1{q^{r-1}(q+1)}
 =
 (-1)^r+
 \sum_{j=3}^{r}
 (-1)^{r+j}\zeta(j-1)
\]
gives Segovia's formula for $\rho_r$.  Equation
\eqref{eq:correction} then follows from
Theorem~\ref{thm:partitiondecomp}.
\end{proof}

Thus the difference between the complete constant $C_r$ and the hook
constant $\rho_r$ has a simple combinatorial meaning.  It is not a
single correction term, but the sum of the contributions associated
with all partitions of $r-1$ having at least two parts.  Equivalently,
it is the total contribution of all non-hook Young-shape families on
the $r$th diagonal.

\subsection{The first complete diagonals}

It is useful to see how the decomposition works in the first cases.
For $r=2$, the only partition of $r-1=1$ is $(1)$, corresponding to
the hook family $(2,1^m)$, and hence
\[
        C_2=C_{(1)}=\rho_2=\frac12.
\]

The first non-hook contribution occurs for $r=3$.  The partitions of
$r-1=2$ are
\[
        (2),\qquad (1,1),
\]
corresponding respectively to the Young-shape families
\[
        (3,1^m),\qquad (2,2,1^m).
\]
Thus
\[
        C_3=C_{(2)}+C_{(1,1)},
\]
where
\[
        C_{(2)}=\rho_3=\frac{\zeta(2)-1}{6},
        \qquad
        C_{(1,1)}=\frac12-\frac{\pi^2}{24},
\]
and therefore
\[
        C_3=\frac13-\frac{\pi^2}{72}.
\]

For $r=4$, the partitions
\[
        (3),\qquad(2,1),\qquad(1,1,1)
\]
give the hook family together with two non-hook families, and their
sum is
\[
        C_4
        =
        \frac14-\frac{\pi^2}{72}
        -\frac{\zeta(3)}{24}.
\]

The case $r=5$ already displays the full structure needed below.  The
partitions of $r-1=4$ are
\[
        (4),\quad(3,1),\quad(2,2),\quad
        (2,1,1),\quad(1,1,1,1),
\]
corresponding to the Young-shape families
\[
\begin{split}
 &(5,1^m),\qquad
 (4,2,1^m),\qquad
 (3,3,1^m),\\
 &(3,2,2,1^m),\qquad
 (2,2,2,2,1^m).
\end{split}
\]
The resulting constant is
\begin{equation}\label{eq:C5}
 C_5
 =
 \frac15-\frac{\pi^2}{80}
 -\frac{31\zeta(3)}{720}
 -\frac{7\pi^4}{86400}.
\end{equation}

What is important here is that the contributions are naturally
grouped not only by their total excess
\[
        |\nu|=r-1,
\]
but also by the number of parts $\ell(\nu)$.  At $r=5$,
\[
\begin{array}{c|c}
\ell(\nu) & \nu \\ \hline
1 &(4)\\
2 &(3,1),(2,2)\\
3 &(2,1,1)\\
4 &(1,1,1,1).
\end{array}
\]
The first row is the hook contribution, while the remaining rows
separate the non-hook terms according to the number of non-unit rows
in the corresponding Young-shape family.  This second statistic is
the grading introduced in the next section; in Section~7 it will
reappear analytically as multiple-zeta depth.
\section{A second grading}
\label{sec:second-grading}

The decomposition of Section~4 is organized by the total excess
\[
        |\nu|=r-1.
\]
The length of the excess partition provides a second natural
statistic.  If
\[
        \nu=(\nu_1,\ldots,\nu_d)\vdash r-1,
\]
then
\[
        d=\ell(\nu)
\]
is the number of non-unit rows in the corresponding Young-shape family
\[
        (\nu_1+1,\ldots,\nu_d+1,1^m).
\]
Thus each contribution $C_\nu$ carries two indices: its total excess
$|\nu|$ and its number of non-unit rows $\ell(\nu)$.

Recall from \eqref{eq:Hcal-ak} that
\[
 \mathcal H(u)
 =
 \prod_{q\geq1}
 \left(
 1+\sum_{k\geq1}a_k(q)u^k
 \right),
\]
where
\[
        a_k(q)
        =
        \frac{1}{(k+1)!\,q^k(q+1)}.
\]
For each $q$, write
\begin{equation}\label{eq:Bq}
 B_q(u)
 :=
 \sum_{k\geq1}a_k(q)u^k.
\end{equation}
Using the exponential series, this can also be written as
\begin{equation}\label{eq:Bq-explicit}
 B_q(u)
 =
 \frac{q(e^{u/q}-1)-u}{u(q+1)},
\end{equation}
where the apparent singularity at $u=0$ is removable and
$B_q(0)=0$.

Selecting the nonconstant term $B_q(u)$ from a factor means that one
non-unit part has been selected.  We may therefore mark each such
selection by a variable $y$.

\begin{definition}
We define the doubly graded generating function
\begin{equation}\label{eq:Hscr}
 \mathscr H(u,y)
 :=
 \prod_{q=1}^\infty
 \left(1+yB_q(u)\right),
\end{equation}
or equivalently,
\[
 \mathscr H(u,y)
 =
 \prod_{q=1}^\infty
 \left[
 1+
 y\,\frac{q(e^{u/q}-1)-u}{u(q+1)}
 \right].
\]
\end{definition}

The two variables now play distinct roles:
\[
 \boxed{
 \begin{array}{ccl}
 u &\text{records}&\text{total excess},\\[1mm]
 y &\text{records}&\text{number of non-unit rows}.
 \end{array}}
\]
Indeed, selecting a term $a_k(q)u^k$ contributes $k$ units of excess
and one power of $y$.

Before extracting coefficients, we record that the infinite product
is well defined.  For $u$ in a compact set,
\[
        B_q(u)=O(q^{-2})
        \qquad(q\to\infty),
\]
uniformly in $u$. Indeed, for $|u|\leq R$ and $|y|\leq M$,
\[
        |yB_q(u)|\leq M C_R q^{-2}.
\]
Hence $\sum_{q\geq1}|yB_q(u)|$ converges uniformly on compact
subsets of $\mathbb C^2$, and the product
$\mathscr H(u,y)$ converges normally there.

\begin{definition}
For $r\geq2$ and $1\leq d\leq r-1$, define
\begin{equation}\label{eq:Crd}
        C_{r,d}
        :=
        [u^{r-1}y^d]\mathscr H(u,y).
\end{equation}
\end{definition}

The coefficient $C_{r,d}$ collects exactly those partitions of
$r-1$ having $d$ parts.

\begin{proposition}[The doubly graded decomposition]
\label{prop:triangle}
For $r\geq2$ and $1\leq d\leq r-1$,
\begin{equation}\label{eq:Crd-Cnu}
        C_{r,d}
        =
        \sum_{\substack{\nu\vdash r-1\\
                        \ell(\nu)=d}}
        C_\nu.
\end{equation}
Consequently,
\begin{equation}\label{eq:row-sum}
        C_r
        =
        \sum_{d=1}^{r-1}C_{r,d}.
\end{equation}
Moreover,
\begin{equation}\label{eq:first-column}
        C_{r,1}=\rho_r.
\end{equation}
\end{proposition}

\begin{proof}
Expand the product \eqref{eq:Hscr}.  A term containing $y^d$ is
obtained by choosing the nonconstant term from exactly $d$ distinct
factors.  If the corresponding powers of $u$ are
\[
        k_1,\ldots,k_d\geq1,
\]
then the total power of $u$ is
\[
        k_1+\cdots+k_d.
\]
Thus a contribution to $u^{r-1}y^d$ is indexed by a partition
$\nu\vdash r-1$ with exactly $d$ parts.  Grouping the terms by this
partition gives \eqref{eq:Crd-Cnu}.

Summing over $d$ recovers all partitions of $r-1$ and therefore,
by Theorem~\ref{thm:partitiondecomp}, gives
\eqref{eq:row-sum}.  Finally, $d=1$ leaves only the one-part
partition $(r-1)$, whose contribution is $\rho_r$ by
Corollary~\ref{cor:hook}.
\end{proof}

The coefficients $C_{r,d}$ can therefore be displayed as a triangular
array:
\[
\begin{array}{c|ccccc}
 & d=1 & d=2 & d=3 & d=4 & \cdots \\ \hline
 r=2 & C_{2,1} \\
 r=3 & C_{3,1} & C_{3,2} \\
 r=4 & C_{4,1} & C_{4,2} & C_{4,3} \\
 r=5 & C_{5,1} & C_{5,2} & C_{5,3} & C_{5,4}\\
 r=6 & C_{6,1} & C_{6,2} & C_{6,3} & C_{6,4}
       & C_{6,5}\\
 \vdots & \vdots & \vdots & \vdots & \vdots & \ddots
\end{array}
\]
The rows reproduce the constants obtained in Section~
\ref{sec:partition-refinement}:
\[
        C_r=C_{r,1}+\cdots+C_{r,r-1}.
\]
The first column is Segovia's hook sequence,
\[
        C_{r,1}=\rho_r,
\]
while the remaining columns separate the non-hook corrections
according to the number of non-unit rows.

For example, the decomposition of $C_5$ obtained in the previous
section becomes
\[
\begin{aligned}
 C_{5,1}
   &=C_{(4)},\\
 C_{5,2}
   &=C_{(3,1)}+C_{(2,2)},\\
 C_{5,3}
   &=C_{(2,1,1)},\\
 C_{5,4}
   &=C_{(1,1,1,1)}.
\end{aligned}
\]
Thus
\[
        C_5
        =
        C_{5,1}+C_{5,2}+C_{5,3}+C_{5,4}.
\]

Setting $y=1$ forgets the number of non-unit rows and recovers the
one-variable product:
\[
        \mathscr H(u,1)=\mathcal H(u).
\]
Thus $\mathscr H(u,y)$ refines $H(u)$ without changing the
constants $C_r$, resolving each of them according to $\ell(\nu)$.

Setting $u=1$, on the other hand, leads to a decomposition by the
number of non-unit rows across all excesses.
Instead of summing across a row of the array, this
amounts to summing down its columns.  At this point the second
grading acquires a meaning that is not apparent from the Young
lattice alone.  Indeed,
\[
 B_q(1)
 =
 \frac{e^{1/q}}{1+1/q}-1,
\]
which is precisely the quantity that appears in the depth generating
function associated with the multiple-zeta-value expansion of
$e^\gamma$.

We postpone the MZV identification itself until
Section~\ref{sec:mzv}, where the necessary convergence issue will be
handled with an Abel regulator.  First, in the next section, we study
the column sums of the array directly.  This already reveals that
the row and column decompositions are two ways of organizing the
same positive total mass.

\section{Column sums and $e^\gamma$}
\label{sec:columns}

The row sums of the array $(C_{r,d})$ recover the fixed-excess
constants $C_r$.  The column decomposition is obtained by setting $u=1$:
\begin{equation}
\mathscr H(1,y)
=
\prod_{q\geq1}(1+yb_q),
\label{eq:Hscr-u1}
\end{equation}
where
\begin{equation}
b_q:=B_q(1)
=
\frac{e^{1/q}}{1+1/q}-1.
\label{eq:bq}
\end{equation}
The numbers $b_q$ are positive.  Indeed,
\[
        e^{1/q}>1+\frac1q,
\]
and hence $b_q>0$.  Moreover,
\[
        b_q
        =
        \frac{1}{2q^2}+O(q^{-3}),
        \qquad q\to\infty,
\]
so that
\[
        \sum_{q\geq1}b_q<\infty.
\]
Consequently, the product \eqref{eq:Hscr-u1} converges absolutely
and its coefficients in $y$ are nonnegative and finite.

\begin{definition}
For $d\geq1$, define the $d$th column mass by
\begin{equation}\label{eq:Sd-def}
        S_d
        :=
        \sum_{r\geq d+1}C_{r,d}.
\end{equation}
We also set $S_0=1$.
\end{definition}

The lower limit $r=d+1$ simply reflects the fact that a partition
of $r-1$ with $d$ positive parts can exist only when $r-1\geq d$.

\begin{theorem}[Column sums]\label{thm:columns}
For every $d\geq1$,
\begin{equation}\label{eq:Sd}
 \begin{aligned}
        S_d
        &=
        [y^d]\mathscr H(1,y)\\
        &=
        \sum_{1\leq q_1<\cdots<q_d}
        b_{q_1}\cdots b_{q_d}.
 \end{aligned}
\end{equation}
Equivalently, $S_d$ is the $d$th elementary symmetric function of
the positive summable sequence
\[
        (b_1,b_2,b_3,\ldots).
\]
\end{theorem}

\begin{proof}
By definition,
\[
        C_{r,d}
        =
        [u^{r-1}y^d]\mathscr H(u,y).
\]
All coefficients of $\mathscr H(u,y)$ are nonnegative.  We may
therefore sum over the $u$-degree and use Tonelli's theorem to obtain
\[
 \sum_{r\geq d+1}C_{r,d}
 =
 [y^d]\mathscr H(1,y).
\]
Expanding
\[
        \mathscr H(1,y)
        =
        \prod_{q\geq1}(1+yb_q)
\]
then gives
\[
 [y^d]\mathscr H(1,y)
 =
 \sum_{1\leq q_1<\cdots<q_d}
 b_{q_1}\cdots b_{q_d},
\]
as claimed.
\end{proof}

Thus the column decomposition has a simple interpretation even before
multiple zeta values enter the picture.  The first column is
\[
        S_1=\sum_{q\geq1}b_q
            =\sum_{r\geq2}C_{r,1}
            =\sum_{r\geq2}\rho_r,
\]
and therefore collects all hook contributions.  The second column
collects all contributions arising from partitions with two non-unit
rows, the third those with three non-unit rows, and so on.

The total mass of all columns can now be evaluated directly.

\begin{theorem}[Total mass]\label{thm:totalmass}
With $S_0=1$,
\begin{equation}\label{eq:egamma-columns}
        \sum_{d\geq0}S_d=e^\gamma.
\end{equation}
Equivalently,
\begin{equation}\label{eq:egamma-array}
        e^\gamma
        =
        1+
        \sum_{d\geq1}
        \sum_{r\geq d+1}C_{r,d}.
\end{equation}
\end{theorem}

\begin{proof}
From \eqref{eq:bq},
\[
        1+b_q
        =
        \frac{e^{1/q}}{1+1/q}
        =
        \frac{q}{q+1}e^{1/q}.
\]
Hence
\[
 \mathscr H(1,1)
 =
 \prod_{q\geq1}(1+b_q)
 =
 \prod_{q\geq1}
 \frac{q}{q+1}e^{1/q}.
\]
For the finite product,
\[
 \prod_{q=1}^n
 \frac{q}{q+1}e^{1/q}
 =
 \frac{e^{H_n}}{n+1}.
\]
Since
\[
        H_n-\log n\longrightarrow\gamma,
\]
we obtain
\[
        \mathscr H(1,1)
        =
        \lim_{n\to\infty}
        \frac{e^{H_n}}{n+1}
        =
        e^\gamma.
\]
On the other hand,
\[
        \mathscr H(1,1)
        =
        \sum_{d\geq0}S_d,
\]
which proves the result.
\end{proof}

Because all the coefficients $C_{r,d}$ are nonnegative, we may now
sum the array in either direction.  Combining
Proposition~\ref{prop:triangle} with
Theorem~\ref{thm:totalmass} gives
\begin{equation}\label{eq:two-sums}
 \boxed{
 \begin{aligned}
 e^\gamma
 &=
 1+\sum_{r\geq2}
       \underbrace{\sum_{d=1}^{r-1}C_{r,d}}_{C_r}\\
 &=
 1+\sum_{d\geq1}
       \underbrace{\sum_{r\geq d+1}C_{r,d}}_{S_d}.
 \end{aligned}}
\end{equation}
The first line reads the triangular array by rows and the second by
columns.

The row index has a meaning inherited from the Young lattice:
$r-1$ is the total excess.  The column index counts the number of
non-unit rows.  In Section~\ref{sec:mzv} we show that the same column index is
the depth in the multiple-zeta-value expansion.

\subsection{The first column masses}

It is instructive to compare the sizes of the first few columns.
Numerically,
\[
\begin{aligned}
 S_1&=0.6359041874\ldots,\\
 S_2&=0.1309505602\ldots,\\
 S_3&=0.0133672211\ldots,\\
 S_4&=8.162652\times10^{-4},\\
 S_5&=3.31983\times10^{-5}.
\end{aligned}
\]
Since $S_0=1$, the first two terms give
\[
        S_0+S_1
        =
        1.635904\ldots,
\]
which is precisely the total obtained from the hook contribution in
Segovia's analysis.  The remaining mass
\[
        e^\gamma-(1+S_1)
\]
is therefore distributed among the columns $d\geq2$.

The rapid decrease visible in these values also explains why the hook
families account for a substantial portion of the total even though
they do not exhaust it. The higher columns successively collect all Young-shape families
having larger numbers of non-unit rows.

At this stage the column decomposition has been obtained entirely
from the positive product $\mathscr H(1,y)$, without invoking multiple
zeta values.

\section{Multiple zeta values and the excess grading}
\label{sec:mzv}

Multiple zeta values already occur in Espinosa's formulation of the
problem.  In particular, Espinosa proposed an Abel-regularized
expansion of $e^\gamma$ organized by depth, and Hucht subsequently
expressed the corresponding depth generating function in terms of
the quantities $b_q$ introduced above~\cite{EspinosaMO2026,HuchtMO2026}.

Our purpose here is not to rederive that decomposition as a new
identity.  Rather, we retain the excess variable $u$ and show that
the resulting refinement is exactly the partition grading developed
in the preceding sections.

We use the ascending-index convention
\begin{equation}\label{eq:mzv-convention}
 \zeta(s_1,\ldots,s_d)
 :=
 \sum_{1\leq q_1<\cdots<q_d}
 \frac{1}
 {q_1^{s_1}\cdots q_d^{s_d}}.
\end{equation}
The integer $d$ is the depth of the multiple zeta value.

\subsection{Why a regulator is needed}

Recall that
\[
        a_k(q)
        =
        \frac{1}{(k+1)!\,q^k(q+1)}.
\]
Writing $a=k+1$, the denominator that must be expanded is
\[
        \frac{1}{q^{a-1}(q+1)}.
\]
Formally,
\[
 \frac{1}{q^{a-1}(q+1)}
 =
 \frac1{q^a}\frac1{1+1/q}
 =
 \sum_{m\geq0}\frac{(-1)^m}{q^{a+m}}.
\]
For $q>1$ this geometric expansion is absolutely convergent, but at
$q=1$ it lies exactly at the endpoint of convergence.  Thus the
formal manipulation suggests the correct MZV expression but does not
by itself justify the required rearrangements.

To justify the rearrangements, introduce an Abel parameter
$0\leq t<1$ and define
\begin{equation}\label{eq:Hcal-t}
 \mathcal H_t(u)
 :=
 \prod_{q=1}^\infty
 \left(
 1+
 \sum_{a\geq2}
 \frac{u^{a-1}}
 {a!\,q^{a-1}(q+t)}
 \right).
\end{equation}
At the endpoint,
\[
        \mathcal H_1(u)=\mathcal H(u).
\]
For $t<1$ we have the absolutely convergent geometric expansion
\begin{equation}\label{eq:geom-reg}
 \frac{1}{q^{a-1}(q+t)}
 =
 \frac1{q^a}\frac1{1+t/q}
 =
 \sum_{m\geq0}
 \frac{(-t)^m}{q^{a+m}}.
\end{equation}
The regulator therefore moves the problematic $q=1$ term into an
absolutely convergent regime.

\begin{lemma}[Absolute convergence with regulator]
\label{lem:regabs}
Fix $0\leq t<1$ and $R>0$.  Then
\begin{equation}\label{eq:regabs}
 \sum_{q\geq1}
 \sum_{a\geq2}
 \sum_{m\geq0}
 \frac{R^{a-1}t^m}
 {a!\,q^{a+m}}
 <\infty.
\end{equation}
Consequently, for $|u|\leq R$, the sums and products obtained by
expanding \eqref{eq:Hcal-t} and \eqref{eq:geom-reg} may be rearranged
absolutely.
\end{lemma}

\begin{proof}
For each $q$,
\[
 \sum_{a\geq2}\sum_{m\geq0}
 \frac{R^{a-1}t^m}
 {a!\,q^{a+m}}
 =
 \frac{1}{1-t/q}
 \sum_{a\geq2}
 \frac{R^{a-1}}{a!\,q^a}.
\]
The term $q=1$ is finite because $t<1$.  For $q\geq2$,
\[
        \frac{1}{1-t/q}\leq2,
\]
while
\[
        \sum_{a\geq2}
        \frac{R^{a-1}}{a!\,q^a}
        =
        O_R(q^{-2}).
\]
The result follows from the convergence of
$\sum_{q\geq2}q^{-2}$.
\end{proof}

We may therefore perform the MZV expansion before taking the endpoint
$t=1$.

\begin{theorem}[Regulated excess-graded MZV expansion]
\label{thm:MZVreg}
For $0\leq t<1$,
\begin{align}
 \mathcal H_t(u)
 =
 1+
 \sum_{d\geq1}
 \sum_{\substack{a_1,\ldots,a_d\geq2\\
                 m_1,\ldots,m_d\geq0}}
 &
 \frac{(-t)^{m_1+\cdots+m_d}}
 {a_1!\cdots a_d!}
 \nonumber\\
 &\times
 \zeta(a_1+m_1,\ldots,a_d+m_d)
 u^{\sum_{j=1}^d(a_j-1)}.
 \label{eq:MZVreg}
\end{align}
For each fixed $t<1$, the expansion is absolutely convergent on
compact subsets of the $u$-plane.
\end{theorem}

\begin{proof}
Insert \eqref{eq:geom-reg} into each factor of
\eqref{eq:Hcal-t}.  Selecting nonconstant terms from exactly $d$
distinct factors, indexed by
\[
        1\leq q_1<\cdots<q_d,
\]
gives
\[
 \prod_{j=1}^d
 \frac{u^{a_j-1}(-t)^{m_j}}
 {a_j!\,q_j^{a_j+m_j}}.
\]
Summing over the ordered indices $q_1<\cdots<q_d$ produces
\[
        \zeta(a_1+m_1,\ldots,a_d+m_d).
\]
Lemma~\ref{lem:regabs} justifies all rearrangements, giving
\eqref{eq:MZVreg}.
\end{proof}

Two gradings are now visible directly in \eqref{eq:MZVreg}.  A term
of depth $d$ contains $d$ integers
\[
        a_1,\ldots,a_d,
\]
whereas its power of $u$ is
\[
        \sum_{j=1}^d(a_j-1).
\]
If we set
\[
        \nu_j=a_j-1,
\]
then
\[
        \nu=(\nu_1,\ldots,\nu_d)
\]
has total size
\[
        |\nu|=\sum_j\nu_j
\]
and length $d$.  Thus the two indices that appeared combinatorially
in the array $C_{r,d}$ arise naturally on the MZV side as well.

\subsection{Removing the regulator}

It remains to justify passage to the endpoint $t=1$.  We do this at
the level of the product, where the convergence is straightforward.

\begin{lemma}[Abel limit]\label{lem:Abel}
As $t\uparrow1$,
\[
        \mathcal H_t(u)\longrightarrow\mathcal H(u)
\]
locally uniformly in $u$.  Consequently, for every fixed $r$,
\begin{equation}\label{eq:Abel-Cr}
        [u^{r-1}]\mathcal H_t(u)
        \longrightarrow C_r.
\end{equation}
\end{lemma}

\begin{proof}
Write
\[
 b_q(t,u)
 :=
 \sum_{a\geq2}
 \frac{u^{a-1}}
 {a!\,q^{a-1}(q+t)}.
\]
For $0\leq t\leq1$ and $u$ in a fixed compact set,
\[
        |b_q(t,u)|\leq Cq^{-2}
\]
with a constant $C$ independent of $q$ and $t$.  Hence the products
\[
        \prod_{q\geq1}(1+b_q(t,u))
\]
converge normally, uniformly in $t\in[0,1]$ on compact $u$-sets.
Since
\[
        b_q(t,u)\longrightarrow b_q(1,u)
\]
for each $q$, the products converge locally uniformly to
$\mathcal H(u)$.  Coefficient convergence follows from Cauchy's
integral formula.
\end{proof}

Combining Theorem~\ref{thm:MZVreg} with
Lemma~\ref{lem:Abel} gives an MZV representation of each fixed-excess
constant.

\begin{corollary}[MZV expansion of a fixed-excess constant]
\label{cor:CrMZV}
For every $r\geq2$,
\begin{align}
 C_r
 =
 \lim_{t\uparrow1}
 \sum_{d=1}^{r-1}
 \sum_{\substack{a_1,\ldots,a_d\geq2\\
                 \sum_j(a_j-1)=r-1}}
 &
 \frac{1}{a_1!\cdots a_d!}
 \nonumber\\
 {}\times
 &
 \sum_{m_1,\ldots,m_d\geq0}
 (-t)^{m_1+\cdots+m_d}
 \zeta(a_1+m_1,\ldots,a_d+m_d).
 \label{eq:CrMZV}
\end{align}
\end{corollary}

The formula becomes more informative if we retain the partition
$\nu$ rather than summing immediately over all partitions of
$r-1$.

\subsection{The MZV block associated with a partition}

Let
\[
        \nu=(\nu_1,\ldots,\nu_d)\vdash r-1,
\]
and let $\operatorname{Perm}(\nu)$ denote the set of distinct
permutations of its parts.  The corresponding Young-shape family is
\[
        (\nu_1+1,\ldots,\nu_d+1,1^m),
        \qquad m\geq0.
\]
On the MZV side, the same partition determines a block of depth $d$.

\begin{theorem}[Partition--MZV correspondence]
\label{thm:CnuMZV}
For every partition
$\nu=(\nu_1,\ldots,\nu_d)\vdash r-1$,
\begin{align}
 C_\nu
 =
 \lim_{t\uparrow1}
 \frac{1}
 {\prod_{j=1}^d(\nu_j+1)!}
 \sum_{\sigma\in\operatorname{Perm}(\nu)}
 \sum_{m_1,\ldots,m_d\geq0}
 &
 (-t)^{m_1+\cdots+m_d}
 \nonumber\\
 {}\times&
 \zeta\!\left(
 \nu_{\sigma(1)}+1+m_1,\ldots,
 \nu_{\sigma(d)}+1+m_d
 \right).
 \label{eq:CnuMZV}
\end{align}
In particular, the length
\[
        d=\ell(\nu)
\]
is the depth of every MZV occurring in the block associated with
$\nu$.
\end{theorem}

The limit above is understood in the Abel sense; no ordinary
convergence of the corresponding unregulated alternating MZV series
at $t=1$ is asserted or required.

\begin{proof}[Proof of Theorem \ref{thm:CnuMZV}]
For $0\leq t<1$, define
\[
 C_\nu(t)
 :=
 \frac{1}
 {\prod_{j=1}^d(\nu_j+1)!}
 \sum_{\sigma\in\operatorname{Perm}(\nu)}
 \sum_{1\leq q_1<\cdots<q_d}
 \prod_{j=1}^d
 \frac{1}
 {q_j^{\nu_{\sigma(j)}}(q_j+t)}.
\]
Applying \eqref{eq:geom-reg} to each factor gives
\begin{align*}
 C_\nu(t)
 =
 \frac{1}
 {\prod_{j=1}^d(\nu_j+1)!}
 \sum_{\sigma\in\operatorname{Perm}(\nu)}
 \sum_{m_1,\ldots,m_d\geq0}
 &
 (-t)^{m_1+\cdots+m_d}\\
 {}\times&
 \zeta\!\left(
 \nu_{\sigma(1)}+1+m_1,\ldots,
 \nu_{\sigma(d)}+1+m_d
 \right),
\end{align*}
where absolute convergence follows from
Lemma~\ref{lem:regabs}.

On the other hand, for $0\leq t\leq1$,
\[
 \frac{1}
 {q_j^{\nu_{\sigma(j)}}(q_j+t)}
 \leq
 \frac{1}
 {q_j^{\nu_{\sigma(j)}+1}}.
\]
Hence
\[
 \prod_{j=1}^d
 \frac{1}
 {q_j^{\nu_{\sigma(j)}}(q_j+t)}
 \leq
 \prod_{j=1}^d
 \frac{1}
 {q_j^{\nu_{\sigma(j)}+1}}.
\]
Since every $\nu_{\sigma(j)}+1\geq2$, the series
\[
 \sum_{1\leq q_1<\cdots<q_d}
 \prod_{j=1}^d
 q_j^{-\nu_{\sigma(j)}-1}
\]
converges absolutely.  Dominated convergence therefore gives
\[
        C_\nu(t)\longrightarrow C_\nu
        \qquad (t\uparrow1),
\]
which proves \eqref{eq:CnuMZV}.
\end{proof}

Theorem~\ref{thm:CnuMZV} identifies the second coordinate of
the triangular array:
\[
 \boxed{
 \begin{array}{rcl}
 r-1=|\nu|
 &\longleftrightarrow&
 \text{Young-lattice excess},\\[1mm]
 d=\ell(\nu)
 &\longleftrightarrow&
 \text{number of non-unit rows}
 \longleftrightarrow
 \text{MZV depth}.
 \end{array}}
\]
Thus a single partition $\nu$ determines simultaneously a family on
a Young-lattice diagonal and a block in the multiple-zeta expansion.

Summing Theorem~\ref{thm:CnuMZV} over all partitions of $r-1$ gives
the row constant $C_r$.  Summing instead over all $r$ while keeping $d$ fixed,
and using the positive column decomposition established in Section~\ref{sec:columns},
gives the column mass $S_d$.  In this precise sense, the fixed-excess
and fixed-depth decompositions are not two unrelated expansions:
they are the two coordinate projections of the same array
$C_{r,d}$.

At $u=1$ the excess grading is forgotten, and the product reduces to
\[
        \mathscr H(1,y)
        =
        \prod_{q\geq1}
        \left[
        1+y\left(
        \frac{e^{1/q}}{1+1/q}-1
        \right)
        \right],
\]
recovering the depth generating function associated with the
Espinosa--Hucht expansion.  The additional variable $u$ is what
retains the Young-lattice excess and resolves each depth contribution
into the fixed-excess components studied here.

The same two-dimensional structure thus arises from both sides:
combinatorially through partitions on fixed Young-lattice diagonals,
and analytically through the depth organization of multiple zeta
values.

\section{Computational verification}
\label{sec:computations}

The results above are exact and do not depend on numerical
computation.  Nevertheless, the equivalent descriptions of the
constants provide useful independent checks of the construction.  We
compare coefficient extraction from the limiting product, the
partition refinement, the double grading, and the finite-$n$
asymptotics.

\subsection{Product and partition expansions}

For a prescribed order $R$, the coefficients
\[
        C_r=[u^{r-1}]\mathcal H(u),
        \qquad 2\leq r\leq R,
\]
can be computed by truncating
\[
 \mathcal H_Q(u)
 =
 \prod_{q=1}^{Q}
 \left(
 1+\sum_{k=1}^{R-1}
 \frac{u^k}{(k+1)!\,q^k(q+1)}
 \right)
 \pmod{u^R}.
\]
The omitted factors differ from $1$ by $O(q^{-2})$ uniformly on
compact $u$-sets, so the coefficients converge rapidly as
$Q\to\infty$.

Symbolic expansion for the first several diagonals agrees exactly
with the independent partition formula
\[
        C_r=\sum_{\nu\vdash r-1}C_\nu.
\]
For example,
\[
 C_5
 =
 C_{(4)}+C_{(3,1)}+C_{(2,2)}
 +C_{(2,1,1)}+C_{(1,1,1,1)},
\]
and both calculations give the value of $C_5$ in
\eqref{eq:C5}.  This agreement checks both the repeated-part
multiplicities in \eqref{eq:Cnu} and the partition indexing of the
diagonal.

The same computation applied to
\[
 \mathscr H(u,y)
 =
 \prod_{q\geq1}
 \left[
 1+
 y\sum_{k\geq1}
 \frac{u^k}{(k+1)!\,q^k(q+1)}
 \right]
\]
constructs the triangular array $(C_{r,d})$.  To every prescribed
finite order it verifies
\[
        \sum_{d=1}^{r-1}C_{r,d}=C_r,
        \qquad
        C_{r,1}=\rho_r.
\]
Setting $u=1$ gives
\[
 \mathscr H(1,y)=\prod_{q\geq1}(1+yb_q),
 \qquad
 b_q=\frac{e^{1/q}}{1+1/q}-1,
\]
so $S_d$ can be computed either by summing the $d$th column
$(C_{r,d})_{r\geq d+1}$ or as the $d$th elementary symmetric
function of $(b_q)_{q\geq1}$.  The two procedures agree, providing an
independent check of the excess--depth grading and of
\[
        \sum_{d\geq0}S_d=e^\gamma.
\]

\subsection{Finite-$n$ asymptotics}

We also tested Theorem~\ref{thm:main} directly from
\[
 D_r(n;z)
 =
 [u^{r-1}]
 \prod_{q=1}^{n}
 \left(
 1+z\frac{e^{u/q}-1}{u}
 \right).
\]
For fixed $r$, it predicts
\[
        \frac{D_r(n;1)}{n}\longrightarrow C_r,
        \qquad
        \frac{A_r(n)}{n\log\log n}\longrightarrow C_r.
\]
Only terms through degree $r-1$ are needed, so the coefficients can
be updated recursively as the factors are multiplied.

\begin{table}[ht]
\centering
\begin{tabular}{c c c c}
\hline
$r$ & $n$ &
$\displaystyle D_r(n;1)/n$ &
$\displaystyle A_r(n)/(n\log\log n)$\\
\hline
3 & 50   & 0.1951970510 & 0.2833411032\\
3 & 200  & 0.1959878101 & 0.2525215491\\
3 & 1000 & 0.1962017916 & 0.2351219256\\
4 & 50   & 0.0621389082 & 0.0973074409\\
4 & 200  & 0.0626605203 & 0.0857797644\\
4 & 1000 & 0.0628011863 & 0.0789918338\\
5 & 50   & 0.0166973227 & 0.0277625236\\
5 & 200  & 0.0169108039 & 0.0243225172\\
5 & 1000 & 0.0169683497 & 0.0222265027\\
\hline
\end{tabular}
\caption{Finite-$n$ values for the unweighted and logarithmically
weighted fixed-excess normalizations.  The first ratio converges
rapidly to $C_r$, while the second exhibits the slower
$\log\log n$ convergence predicted by Theorem~\ref{thm:main}.}
\label{tab:finite-n}
\end{table}

The distribution of the number of parts provides a further check.
With
\[
 \mathbb P(I_{n,r}=i)
 =
 \frac{E_{i,i+r-1}(n)}{D_r(n;1)},
\]
the computations give
\[
 \mathbb E[I_{n,r}]-\log n=O(1),
 \qquad
 \operatorname{Var}(I_{n,r})-\log n=O(1),
\]
in agreement with Lemma~\ref{lem:concentration}.  Thus the numerical
data display the two scales underlying Theorem~\ref{thm:main}: linear
growth of the diagonal mass and concentration of the number of parts
near $\log n$.

All these checks use only truncated power series and finite products
and can be reproduced in a standard computer algebra system.  They
give three independent routes to the same constants: coefficient
extraction from $\mathcal H(u)$, summation of the $C_\nu$, and the
finite-$n$ limit $D_r(n;1)/n$.

\section{Discussion and outlook}
\label{sec:discussion}

The main outcome of this work is a unified organization of the
partition structures arising on fixed-excess Young-lattice
diagonals.  Rather than treating the non-hook terms as isolated
corrections to Segovia's hook families, the generating-function
framework shows that they are canonically indexed by the partitions
\[
        \nu\vdash r-1,
\]
with
\[
        C_r=\sum_{\nu\vdash r-1}C_\nu.
\]
Retaining the length of $\nu$ produces the finer array
\[
        C_{r,d}
        =
        [u^{r-1}y^d]\mathscr H(u,y),
\]
which is the central organizing object of the paper.  Its two
coordinates record, respectively, the total Young-lattice excess
$r-1$ and the number $d$ of non-unit rows.

The same second coordinate acquires an analytic meaning under the
Abel-regularized multiple-zeta correspondence: each partition $\nu$
determines an MZV block of depth
\[
        d=\ell(\nu).
\]
Consequently, the row and column sums of the same positive array give
the fixed-excess and fixed-depth decompositions,
\[
        C_r=\sum_d C_{r,d},
        \qquad
        S_d=\sum_r C_{r,d},
\]
and hence
\[
 e^\gamma
 =
 1+\sum_{r\geq2}\sum_{d=1}^{r-1}C_{r,d}
 =
 1+\sum_{d\geq1}\sum_{r\geq d+1}C_{r,d}.
\]
The appearance of $e^\gamma$ and its MZV depth decomposition is
consistent with the earlier work of Espinosa and Hucht and is not
claimed here as a new identity.  What the present construction adds
is the excess grading: it resolves the depth decomposition according
to Young-lattice diagonals and, conversely, each fixed diagonal
according to MZV depth.

Several directions remain open.  First, the marked asymptotic
\[
        D_r(n;z)\sim n^zK_r(z)
\]
contains substantially more information than is needed for the
fixed-excess limit.  Its form suggests a mod-Poisson limit theorem
for the number of parts and, in particular, Gaussian fluctuations
around $\log n$ on the scale $\sqrt{\log n}$.  Establishing such a
limit, together with higher-order asymptotic expansions for
$A_r(n)$, would refine the first-order result obtained here.

A second direction concerns the arithmetic structure of the
constants $C_\nu$ and $C_{r,d}$.  Already at low excess they involve
zeta values and their products, while the regulated formulation
places them naturally in the setting of multiple zeta values.  It
would be interesting to determine which reductions are available at
fixed excess or fixed depth and whether the triangular organization
reveals identities that are less visible in the usual MZV grading.

Finally, throughout this paper the excess $r-1$ is fixed while $n$
tends to infinity.  The coefficients
\[
        C_r=[u^{r-1}]\mathcal H(u)
\]
themselves form a sequence whose behavior as $r\to\infty$ is not
addressed here.  Understanding this regime, as well as joint limits
in which $r$ grows with $n$, could reveal a different part of the
partition geometry underlying the original expansion.

More broadly, the analysis illustrates the advantage of passing from
individual partition shapes to a generating object that retains the
relevant statistics simultaneously.  A shape-by-shape calculation
on the Young lattice becomes a two-dimensional structure in which
partition excess, non-unit rows, asymptotic mass, and MZV depth are
encoded by the same coefficients.  It is this unified organization,
rather than any new claim concerning the Riemann hypothesis itself,
that is the principal outcome of the paper.

\section*{Acknowledgments}
This work was partly supported by DGAPA-PAPIIT grant number IN112725.

\section*{Data availability}

No datasets were generated or analyzed in this study.
The numerical computations reported in Section~\ref{sec:computations}
are reproducible directly from the explicit formulas given in the text.

\section*{Declaration of AI-assisted technologies}

During the development of this work, the author used ChatGPT
(OpenAI) as an AI-assisted research and writing tool.  It was used
for mathematical exploration, checking and refining arguments,
computational verification, literature exploration, and assistance
with the organization and editing of the manuscript.  All mathematical
statements, proofs, computations, references, and final wording were
independently reviewed and verified by the author, who takes full
responsibility for the content of the manuscript.

\end{document}